\documentclass{birkjour}
\usepackage{amssymb}

\usepackage{hyperref} 

\usepackage{url}      

\newcommand{\cM}{\mathcal{M}}
\newcommand{\cP}{\mathcal{P}}

\newcommand{\N}{\mathbb{N}}
\newcommand{\R}{\mathbb{R}}

\newcommand{\Op}{\operatorname{Op}}

\newtheorem{theorem}{Theorem}[section]
\newtheorem{lemma}[theorem]{Lemma}
\newtheorem{corollary}[theorem]{Corollary}
\theoremstyle{definition}
\newtheorem{definition}[theorem]{Definition}

\title[Boundedness of PDO on Banach Function Spaces]
{Boundedness of Pseudodifferential Operators\\ on Banach Function Spaces}

\author[A. Yu. Karlovich]{Alexei Yu. Karlovich}
\address{
Departamento de Matem\'atica\\
Faculdade de Ci\^encias e Tecnologia\\
Universidade Nova de Lisboa\\
Quinta da Torre\\
2829--516 Caparica\\
Portugal} \email{oyk@fct.unl.pt}

\thanks{The author is partially supported by FCT project PEstOE/MAT/UI4032/2011 (Portugal).}
\dedicatory{To Professor Ant\'onio Ferreira dos Santos}

\begin{document}

\begin{abstract}
We show that if the Hardy-Littlewood maximal operator is bounded on a
separable Banach function space $X(\mathbb{R}^n)$ and on its associate space
$X'(\mathbb{R}^n)$, then a pseudodifferential operator $\operatorname{Op}(a)$ is bounded
on $X(\mathbb{R}^n)$ whenever the symbol $a$ belongs to the H\"ormander class
$S_{\rho,\delta}^{n(\rho-1)}$ with $0<\rho\le 1$, $0\le\delta<1$
or to the the Miyachi class $S_{\rho,\delta}^{n(\rho-1)}(\varkappa,n)$
with $0\le\delta\le\rho\le 1$, $0\le\delta<1$,  and $\varkappa>0$.
This result is applied to the case of variable Lebesgue spaces
$L^{p(\cdot)}(\mathbb{R}^n)$.
\end{abstract}

\maketitle

\section{Introduction}
We denote the usual operators of first order partial
differentiation on $\R^n$ by
$\partial_{x_j}:=\partial/\partial_{x_j}$. For every multi-index
$\alpha=(\alpha_1,\dots,\alpha_n)$ with non-negative integers
$\alpha_j$, we write
$\partial^\alpha:=\partial_{x_1}^{\alpha_1}\dots\partial_{x_n}^{\alpha_n}$.
Further, put $|\alpha|:=\alpha_1+\dots+\alpha_n$, and for each vector
$\xi=(\xi_1,\dots,\xi_n)\in\R^n$, define
$\xi^\alpha:=\xi_1^{\alpha_1}\dots\xi_n^{\alpha_n}$.
Let $\langle\cdot,\cdot\rangle$ stand for the scalar product in $\R^n$
and $|\xi|:=\sqrt{\langle\xi,\xi\rangle}$ for $\xi\in\R^n$.

Let $C_0^\infty(\R^n)$ denote the set of all infinitely
differentiable functions with compact support. Recall that, given
$u\in C_0^\infty(\R^n)$, a pseudodifferential operator $\Op(a)$ is
formally defined by the formula
\[
(\Op(a)u)(x):=\frac{1}{(2\pi)^n}\int_{\R^n}d\xi\int_{\R^n}a(x,\xi)u(y)e^{i\langle x-y,\xi\rangle}dy,
\]
where the symbol $a$ is assumed to be bounded in both the spatial variable $x$ and
the frequency variable $\xi$, and satisfies certain regularity conditions.

An example of symbols one might consider is the H\"ormander class $S_{\rho,\delta}^m$
introduced in \cite{H67} and consisting of $a\in C^\infty(\R^n\times\R^n)$ with
\[
|\partial_\xi^\alpha\partial_x^\beta a(x,\xi)|\le C_{\alpha,\beta}(1+|\xi|)^{m-\rho|\alpha|+\delta|\beta|}
\quad (x,\xi\in\R^n),
\]
where
\[
m\in\R,\quad 0\le\delta,\rho\le 1
\]
and the positive constants $C_{\alpha,\beta}$ depend only on $\alpha$ and $\beta$.
Along with the H\"ormander class $S_{\rho,\delta}^m$, we will consider the
generalized H\"ormander class $S_{\rho,\delta}^m(\varkappa,\varkappa')$
introduced by Miyachi \cite{M88}. We will call $S_{\rho,\delta}^m(\varkappa,\varkappa')$
the Miyachi class of symbols. Its quite technical definition is postponed
to Subsection~\ref{subsec:Miyachi}. Here we only note that symbols in the Miyachi classes
may lie beyond $C^{\infty}(\R^n\times\R^n)$ (that is, they are non-smooth, in general).

Let $f\in L^1_{\rm loc}(\mathbb{R}^n)$.  For a cube
$Q\subset\mathbb{R}^n$, put
\[
f_Q:=\frac{1}{|Q|}\int_Q f(x)dx.
\]
Here, and throughout, cubes will be assumed to have their sides parallel to the
coordinate axes and $|Q|$ will denote the volume of $Q$. The Fefferman-Stein
sharp maximal operator $f\mapsto f^\#$ is defined by
\[
f^\#(x):=\sup_{Q\ni x}\frac{1}{|Q|}\int_Q|f(x)-f_Q|dx
\quad (x\in\R^n),
\]
where the supremum is taken over all cubes $Q$ containing $x$. Let $1\le q<\infty$.
Given $f\in L_{\rm loc}^q(\R^n)$, the $q$-th maximal operator is defined by
\[
(M_qf)(x):=\sup_{Q\ni x}\left(\frac{1}{|Q|}\int_Q|f(y)|^q dy\right)^{1/q}\quad (x\in\R^n),
\]
where the supremum is taken over all cubes $Q$ containing $x$.
For $q=1$ this is the usual Hardy-Littlewood maximal operator, which will be denoted by $M$.

The boundedness of pseudodifferential operators with smooth and non-smooth symbols
on the classical Lebesgue spaces $L^p(\R^n)$ was studied by many authors.
We refer to the monographs by
Coifman and Meyer \cite{CM78},
Kumano-go \cite{K82},
Journ\'e \cite{J83},
Taylor \cite{T91},
Stein \cite{S93},
H\"ormander \cite{H07},
Abels \cite{A12}
and also to the papers by
Miyachi \cite{M88} and
Ashino, Nagase, and Vaillancourt \cite{ANV04}
for corresponding results and further references.

Miller \cite{M82} proved the boundedness of pseudodifferential operators with
symbols $a\in S_{1,0}^0$ on the weighted Lebesgue spaces $L^p(\R^n,w)$
with $1<p<\infty$ and Muckenhoupt weights $w\in A_p(\R^n)$. One of the key
ingredients in his  proof was the pointwise estimate
\begin{equation}\label{eq:pointwise}
(\Op(a)f)^\#(x)\le C_q (M_qf)(x)\quad (x\in\R^n),
\end{equation}
where $q\in(1,\infty)$ and $C_q>0$ is independent of $f\in C_0^\infty(\R^n)$.
Another ingredients are the Fefferman-Stein inequality (see e.g. \cite[Theorem~5]{FS72})
and self-improving properties of Muckenhoupt weights.
Further, estimate \eqref{eq:pointwise} and the boundedness results
for $\Op(a)$ on $L^p(\R^n,w)$ with $p\in(1,\infty)$ and $w\in A_p(\R^n)$ were
extended to other classes of smooth and non-smooth symbols. We refer, for
instance, to the works by
Nishigaki \cite{N84},
Yabuta \cite{Y85a,Y85b,Y86,Y89},
Miyachi and Yabuta \cite{MY87},
\'Alvarez and Hounie \cite{AH90},
\'Alvarez, Hounie, and P\'erez \cite{AHP91},
Michalowski, Rule, and Staubach \cite{MRS12}
and the references therein.

Rabinovich and Samko \cite[Theorem~5.1]{RS08} proved the boundedness of pseudodifferential
operators with symbols $a\in S_{1,0}^0$ on so-called variable Lebesgue spaces $L^{p(\cdot)}(\R^n)$
(see Subsection~\ref{subsec:VLE}). Their proof did not rely on \eqref{eq:pointwise}.
Instead, they obtained another (more precise) pointwise estimate for $(\Op(a)f)^\#(x)$
in the spirit of \cite{AP94}. Recently the author and Spitkovsky \cite[Theorem~1.2]{KS13a}
proved the boundedness of $\Op(a)$ on variable Lebesgue spaces $L^{p(\cdot)}(\R^n)$
for the symbols $a\in S_{\rho,\delta}^{n(\rho-1)}$ with $0<\rho\le 1$ and
$0\le\delta<1$. That proof relies on \eqref{eq:pointwise} (obtained in
\cite{MRS12}), on the Fefferman-Stein inequality for variable Lebesgue spaces,
and on a certain self-improving property of the Hardy-Littlewood maximal
function on $L^{p(\cdot)}(\R^n)$.

The aim of the present paper is to extend the results of \cite{KS13a,RS08} to the
case of so-called Banach function spaces. Our proof is based on estimate
\eqref{eq:pointwise}, on the Fefferman-Stein inequality for Banach function
spaces proved recently by Lerner \cite{L10}, and on a self-improving property
of the Hardy-Littlewood maximal function on Banach function spaces
proved by Lerner and P\'erez \cite{LP07}. Note that our results are true for
all symbols classes admitting estimate \eqref{eq:pointwise}. We choose
here the classical H\"ormander classes $S_{\rho,\delta}^m$ of smooth symbols
and the Miyachi classes $S_{\rho,\delta}^m(\varkappa,\varkappa')$ of non-smooth
symbols just as an illustration of the fact that the assumptions on smoothness
of symbols imposed
in \cite{KS13a,RS08} can be essentially relaxed.

The set of all Lebesgue measurable complex-valued functions on $\R^n$ is
denoted by $\cM$. Let $\cM^+$ be the subset of functions in $\cM$ whose
values lie  in $[0,\infty]$. The characteristic function of a measurable
set $E\subset\R^n$ is denoted by $\chi_E$ and the Lebesgue measure of $E$
is denoted by $|E|$.
\begin{definition}[{\cite[Chap.~1, Definition~1.1]{BS88}}]
\label{def-BFS}
A mapping $\rho:\cM^+\to [0,\infty]$ is
called a {\it Banach function norm} if, for all functions $f,g,
f_n \ (n\in\N)$ in $\cM^+$, for all constants $a\ge 0$, and for
all measurable subsets $E$ of $\R^n$, the following properties
hold:
\begin{eqnarray*}
{\rm (A1)} & & \rho(f)=0  \Leftrightarrow  f=0\ \mbox{a.e.}, \quad
\rho(af)=a\rho(f), \quad
\rho(f+g) \le \rho(f)+\rho(g),\\
{\rm (A2)} & &0\le g \le f \ \mbox{a.e.} \ \Rightarrow \ \rho(g)
\le \rho(f)
\quad\mbox{(the lattice property)},\\
{\rm (A3)} & &0\le f_n \uparrow f \ \mbox{a.e.} \ \Rightarrow \
       \rho(f_n) \uparrow \rho(f)\quad\mbox{(the Fatou property)},\\
{\rm (A4)} & & |E|<\infty \Rightarrow \rho(\chi_E) <\infty,\\
{\rm (A5)} & & |E|<\infty \Rightarrow \int_E f(x)\,dx \le C_E\rho(f)
\end{eqnarray*}
with $C_E \in (0,\infty)$ which may depend on $E$ and $\rho$ but is
independent of $f$.
\end{definition}
When functions differing only on a set of measure zero are identified,
the set $X(\R^n)$ of all functions $f\in\cM$ for which $\rho(|f|)<\infty$ is
called a \textit{Banach function space}. For each $f\in X(\R^n)$, the norm of
$f$ is defined by
\[
\|f\|_{X(\R^n)} :=\rho(|f|).
\]
The set $X(\R^n)$ under the natural linear space operations and under this norm
becomes a Banach space (see \cite[Chap.~1, Theorems~1.4 and~1.6]{BS88}).

If $\rho$ is a Banach function norm, its associate norm $\rho'$ is
defined on $\cM^+$ by
\[
\rho'(g):=\sup\left\{
\int_{\R^n} f(x)g(x)\,dx \ : \ f\in \cM^+, \ \rho(f) \le 1
\right\}, \quad g\in \cM^+.
\]
It is a Banach function norm itself \cite[Chap.~1, Theorem~2.2]{BS88}.
The Banach function space $X'(\R^n)$ determined by the Banach function norm
$\rho'$ is called the \textit{associate space} (\textit{K\"othe dual}) of $X(\R^n)$.
The Lebesgue space $L^p(\R^n)$, $1\le p\le\infty$, are the the archetypical
example of Banach function spaces. Other classical examples of Banach function
spaces are Orlicz spaces, rearrangement-invariant spaces, and variable Lebesgue
spaces $L^{p(\cdot)}(\R^n)$.

Note that we do not assume that $X(\R^n)$ is rearrangement-invariant
(see \cite[Chap.~2]{BS88}). Therefore, we are not allowed to use the interpolation
theory to study the boundedness of $\Op(a)$ on $X(\R^n)$.
\begin{theorem}[Main result]
\label{th:main}
Let $X(\R^n)$ be a separable Banach function space such that the Hardy-Littlewood
maximal operator $M$ is bounded on $X(\R^n)$ and on its associate space $X'(\R^n)$.
If $a$ belongs to one of the following symbol classes:
\begin{enumerate}
\item[{\rm(a)}]
the H\"ormander class $S_{\rho,\delta}^{n(\rho-1)}$ with $0<\rho\le 1$ and $0\le\delta<1$;

\item[{\rm(b)}]
the Miyachi class $S_{\rho,\delta}^{n(\rho-1)}(\varkappa,n)$ with
$0\le\delta\le\rho\le 1$, $0\le\delta<1$,  and $\varkappa>0$;
\end{enumerate}
then $\Op(a)$ extends to a bounded operator on $X(\R^n)$.
\end{theorem}
The paper is organized as follows. Section~\ref{sec:proof} is devoted to the
proof of Theorem~\ref{th:main}. First, we collect the main ingredients. We
give the precise definition of the Miyachi class
$S_{\rho,\delta}^m(\varkappa,\varkappa')$ in Subsection~\ref{subsec:Miyachi}.
The Fefferman-Stein inequality for Banach function spaces is stated in
Section~\ref{subsec:Fefferman-Stein}. A certain self-improving property of
the Hardy-Littlewood maximal operator on Banach function spaces is discussed in
Subsection~\ref{subsec:self-improving}. Precise assumptions on our symbols
guaranteeing \eqref{eq:pointwise} are stated in
Subsection~\ref{subsec:pointwise}. Finally, we assemble these ingredients in
Subsection~\ref{subsec:proof} and prove Theorem~\ref{th:main}.

In Section~\ref{sec:VLE} we apply Theorem~\ref{th:main} to the case of
variable Lebesgue spaces $L^{p(\cdot)}(\R^n)$. In Subection~\ref{subsec:VLE}
we recall the definition and some basic properties of variable Lebesgue spaces.
In Subsection~\ref{subsec:M-VLE} we discuss the boundedness of the Hardy-Littlewood
maximal operator on $L^{p(\cdot)}(\R^n)$. In particular, we recall that
$M$ is bounded on $L^{p(\cdot)}(\R^n)$ if and only if $M$ is bounded on its
associate space. This allows us to simplify little bit the formulation
of Theorem~\ref{th:main} for $L^{p(\cdot)}(\R^n)$ in Subsection~\ref{subsec:PDO-VLE}.
\section{Proof of the main result}\label{sec:proof}
\subsection{The Miyachi class}\label{subsec:Miyachi}
The following class of symbols was introduced by Miyachi \cite{M88} (see also
\cite{M87,MY87}).
If $h\in\R^n$ and $f$ is a function on $\R^n$, then the first and the second
differences are denoted by
\begin{align*}
\Delta_x(h)f(x) &:=f(x+h)-f(x),
\\
\Delta_x^2(h)f(x) &:=f(x+2h)-2f(x+h)+f(x).
\end{align*}
Let
\[
m\in\R,\quad 0\le\delta,\rho\le 1,\quad\varkappa>0,\quad \varkappa'>0.
\]
Let $k$ and $k'$ be nonnegative integers satisfying
\[
k<\varkappa\le k+1,\quad k'<\varkappa '\le k'+1.
\]
The Miyachi class $S_{\rho,\delta}^m(\varkappa,\varkappa')$
consists of all functions $a$ on $\R^n\times\R^n$ such that the
derivatives $\partial_x^\beta\partial_\xi^\alpha a(x,\xi)$ exist
in the classical sense for $|\beta|\le k$ and $|\alpha|\le k'$ and the following
four conditions are fulfilled:

\begin{enumerate}
\item[(i)]
if $|\beta|\le k$ and $|\alpha|\le k'$, then
\[
|\partial_x^\beta\partial_\xi^\alpha a(x,\xi)|\le A(1+|\xi|)^{m+\delta|\beta|-\rho|\alpha|};
\]

\item[(ii)]
if $|\beta|=k$ and $|\alpha|\le k'$, $h\in\R^n$, and $|h|\le(1+|\xi|)^{-\delta}$, then
\[
|\Delta_x^2(h)\partial_x^\beta\partial_\xi^\alpha a(x,\xi)|
\le
A(1+|\xi|)^{m+\delta\varkappa-\rho|\alpha|}|h|^{\varkappa-k};
\]

\item[(iii)]
if $|\beta|\le k$ and $|\alpha|=k'$, $\eta\in\R^n$ and $|\eta|\le(1+|\xi|)^\rho/4$, then
\[
|\Delta_\xi^2(\eta)\partial_x^\beta\partial_\xi^\alpha a(x,\xi)|
\le
A(1+|\xi|)^{m+\delta|\beta|-\rho\varkappa'}|\eta|^{\varkappa'-k'};
\]

\item[(iv)]
if $|\beta|=k$ and $|\alpha|=k'$,
$h,\eta\in\R^n$, and $|h|\le(1+|\xi|)^{-\delta}$, $|\eta|\le(1+|\xi|)^\rho/4$, then
\[
|\Delta_x^2(h)\Delta_\xi^2(\eta)\partial_x^\beta\partial_\xi^\alpha a(x,\xi)|
\le
A(1+|\xi|)^{m+\delta\varkappa-\rho\varkappa'}|h|^{\varkappa-k}|\eta|^{\varkappa'-k'}.
\]
\end{enumerate}
Here the constant $A$ is independent of the multi-indices $\alpha$, $\beta$ and
the variables $x,\xi,h,\eta\in\R^n$. The smallest such constant is denoted by
$\|a\|_{m,\rho,\delta,\varkappa,\varkappa'}$.

It is not difficult to see that if $\varkappa_2\le\varkappa_1$ and $\varkappa_2'\le\varkappa_1'$, then
\[
S_{\rho,\delta}^m\subset S_{\rho,\delta}^m(\varkappa_1,\varkappa_1')\subset S_{\rho,\delta}^m(\varkappa_2,\varkappa_2')
\]
and
\[
\|a\|_{m,\rho,\delta,\varkappa_2,\varkappa_2'}
\le
{\rm const}\|a\|_{m,\rho,\delta,\varkappa_1,\varkappa_1'}.
\]
If $\varkappa$ (resp. $\varkappa'$) is not integer, then $\Delta_x^2(h)$ (resp.
$\Delta_\xi^2(\eta)$) can be replaced by $\Delta_x^1(h)$ (resp. $\Delta_\xi^1(\eta)$).
It should also be remarked that the assumptions $|h|\le(1+|\xi|)^{-\delta}$
and $|\eta|\le(1+|\xi|)^\rho/4$ can be replaced by $h\in\R^n$ and $|\eta|\le(1+|\xi|)/4$
if one modifies the constant $A$.
\subsection{Density of smooth compactly supported functions}
\begin{lemma}\label{le:density}
The set $C_0^\infty(\R^n)$ is dense in a separable Banach function space $X(\R^n)$.
\end{lemma}
The proof is standard. For details, see \cite[Lemma~2.10(b)]{KS13b}, where
this fact is proved for $n=1$. The proof for arbitrary $n$ is a minor modification
of that one.
\subsection{The Fefferman-Stein inequality for Banach function spaces}
\label{subsec:Fefferman-Stein}
Let $S_0(\R^n)$ be the space of all measurable functions $f$ on $\R^n$ such that
\[
|\{x\in\R^n:|f(x)|>\lambda\}|<\infty
\]
for any $\lambda>0$. Chebyshev's inequality
\[
|\{x\in\R^n:|f(x)|>\lambda\}|\le\frac{1}{\lambda^q}\int_{\R^n}|f(x)|^q\,dx
\]
holds for every $q\in(0,\infty)$ and $\lambda>0$. In particular, it implies that
\[
\bigcup_{q\in(0,\infty)}L^q(\R^n)\subset S_0(\R^n).
\]

It is obvious that $f^\#$ is pointwise dominated by $Mf$. Hence,
by Axiom (A2),
\[
\|f^\#\|_{X(\R^n)}\le {\rm const}\|f\|_{X(\R^n)}
\quad\mbox{for}\quad f\in X(\R^n)
\]
whenever $M$ is bounded on $X(\R^n)$. The converse inequality for Lebesgue spaces
$L^p(\R^n)$, $1<p<\infty$, was proved by Fefferman and Stein (see
\cite[Theorem~5]{FS72} and also
\cite[Chap.~IV, Section~2.2]{S93}). The following extension of the Fefferman-Stein
inequality to Banach function spaces was proved in \cite[Corollary~4.2]{L10}.
\begin{theorem}[Lerner]
\label{th:Lerner}
Let $M$ be bounded on a Banach function space $X(\R^n)$. Then $M$ is bounded
on its associate space $X'(\R^n)$ if and only
if there exists a constant $C_\#>0$ such that, for all $f\in S_0(\R^n)$,
\[
\|f\|_{X(\R^n)}\le C_\#\|f^\#\|_{X(\R^n)}.
\]
\end{theorem}
\subsection{Self-improving property of maximal operators on Banach function spaces}
\label{subsec:self-improving}
If $1<q<\infty$, then from the H\"older inequality one can immediately get that
\[
(Mf)(x)\le (M_qf)(x)\quad (x\in\R^n).
\]
Thus, the boundedness of any $M_q$, $1<q<\infty$, on a Banach function space
$X(\R^n)$  immediately implies the boundedness of $M$. A partial converse
of this fact, called a \textit{self-improving property} of the Hardy-Littlewood
maximal operator, is also true. It was proved in
\cite[Corollary~1.3]{LP07} (see also \cite{LO10} for another proof)
in a more general setting of quasi-Banach function spaces.
\begin{theorem}[Lerner-P\'erez]
\label{th:Lerner-Perez}
Let $X(\R^n)$ be a Banach function space. Then $M$ is bounded on $X(\R^n)$ if
and only if $M_q$ is bounded on $X(\R^n)$ for some $q\in(1,\infty)$.
\end{theorem}
\subsection{The crucial pointwise estimate}\label{subsec:pointwise}
\begin{theorem}\label{th:pointwise}
If $a$ belongs to one of the following symbol classes:
\begin{enumerate}
\item[{\rm(a)}]
the H\"omander class
$S_{\rho,\delta}^{n(\rho-1)}$ with $0<\rho\le 1$ and $0\le\delta<1$;

\item[{\rm(b)}]
the Miyachi class
$S_{\rho,\delta}^{n(\rho-1)}(\varkappa,n)$ with $0\le\delta\le\rho\le 1$, $0\le\delta<1$,  and $\varkappa>0$;
\end{enumerate}
then for every $q\in(1,\infty)$ there exists a constant $C_q>0$ such that
\begin{equation}\label{eq:pointwise-2}
(\Op(a)f)^\#(x)\le C_q (M_qf)(x)\quad (x\in\R^n)
\end{equation}
for all $f\in C_0^\infty(\R^n)$.
\end{theorem}
Part (a) was recently proved by Michalowski, Rule, and Staubach \cite[Theorem~3.3]{MRS12}.
Their estimate generalizes the pointwise estimate by Miller
\cite[Theorem~2.8]{M82} for $a\in S_{1,0}^0$ and by \'Alvarez and
Hounie \cite[Theorem~4.1]{AH90} for $a\in S_{\rho,\delta}^m$ with
the parameters satisfying $0<\delta\le\rho\le 1/2$ and $m\le n(\rho-1)$.
Part (b) follows from the estimate by Miyachi and Yabuta \cite[Theorem~2.4]{MY87}.
\begin{corollary}\label{co:So}
If the conditions of Theorem~\ref{th:pointwise} are fulfilled, then
$\Op(a)f\in S_0(\R^n)$ for every $f\in C_0^\infty(\R^n)$.
\end{corollary}
\begin{proof}
By using the well-known $L^p$-estimates for the sharp maximal function
(see \cite[Theorem~5]{FS72}) and for the maximal function $M_q$,
one can show that if \eqref{eq:pointwise-2} holds for all $f\in C_0^\infty(\R^n)$,
then $\Op(a)$ extends to a bounded operator on $L^p(\R^n)$ for $q<p<\infty$.
In particular, this implies that $\Op(a)f\in S_0(\R^n)$ for every function
$f\in C_0^\infty(\R^n)$.
\end{proof}
\subsection{Proof of Theorem~\ref{th:main}}\label{subsec:proof}
The presented proof is an adaptation of the proof of \cite[Theorem~1.2]{KS13a}.
Its idea goes back to Miller \cite{M82}.
Suppose $f\in C_0^\infty(\R^n)$. Then $\Op(a)f\in S_0(\R^n)$ in view of Corollary~\ref{co:So}.
By Lerner's theorem (Theorem~\ref{th:Lerner}), there exists a constant
$C_\#>0$ such that
\begin{equation}\label{eq:proof-1}
\|\Op(a)f\|_{X(\R^n)}\le C_\#\|(\Op(a)f)^\#\|_{X(\R^n)}
\end{equation}
Further, by the crucial pointwise estimate (Theorem~\ref{th:pointwise}),
for every $q\in(1,\infty)$, there is a constant $C_q>0$ such that
\[
(\Op(a)f)^\#(x)\le C_q (M_qf)(x)\quad (x\in\R^n)
\]
Hence, by Axioms (A1) and (A2),
\begin{equation}\label{eq:proof-2}
\|(\Op(a)f)^\#\|_{X(\R^n)}\le C_q\|M_qf\|_{X(\R^n)}.
\end{equation}
On the other hand, since $M$ is bounded on $X(\R^n)$, by the Lerner-P\'erez
theorem (Theorem~\ref{th:Lerner-Perez}), there is a constant exponent
$q_0\in (1,\infty)$ and a constant $C_{q_0}'>0$ such that
\begin{equation}\label{eq:proof-3}
\|M_{q_0}f\|_{X(\R^n)}\le C_{q_0}'\|f\|_{X(\R^n)}.
\end{equation}
Thus, combining \eqref{eq:proof-1}--\eqref{eq:proof-3}, we arrive at
\[
\|\Op(a)f\|_{X(\R^n)}\le C_\#C_{q_0}C_{q_0}'\|f\|_{X(\R^n)}
\]
for all $f\in C_0^\infty(\R^n)$.
It remains to recall that, in view of Lemma~\ref{le:density}, $C_0^\infty(\R^n)$
is dense in $X(\R^n)$ whenever $X(\R^n)$ is separable.
Thus, $\Op(a)$ extends to a bounded operator on the whole space
$X(\R^n)$ by continuity.
\qed
\section{Pseudodifferential operators on variable Lebesgue spaces}\label{sec:VLE}
\subsection{Variable Lebesgue spaces}\label{subsec:VLE}
Let $p\colon\R^n\to[1,\infty]$ be a measurable a.e. finite function. By
$L^{p(\cdot)}(\R^n)$ we denote the set of all complex-valued functions
$f$ on $\R^n$ such that
\[
I_{p(\cdot)}(f/\lambda):=\int_{\R^n} |f(x)/\lambda|^{p(x)} dx <\infty
\]
for some $\lambda>0$. This set becomes a Banach function space when
equipped with the norm
\[
\|f\|_{p(\cdot)}:=\inf\big\{\lambda>0: I_{p(\cdot)}(f/\lambda)\le 1\big\}.
\]
It is easy to see that if $p$ is constant, then $L^{p(\cdot)}(\R^n)$ is nothing but
the standard Lebesgue space $L^p(\R^n)$. The space $L^{p(\cdot)}(\R^n)$
is referred to as a \textit{variable Lebesgue space}.

We will always suppose that
\begin{equation}\label{eq:exponents}
1<p_-:=\operatornamewithlimits{ess\,inf}_{x\in\R^n}p(x),
\quad
\operatornamewithlimits{ess\,sup}_{x\in\R^n}p(x)=:p_+<\infty.
\end{equation}
Under these conditions, the space $L^{p(\cdot)}(\R^n)$ is
separable and reflexive, and its associate space is isomorphic to
$L^{p'(\cdot)}(\R^n)$, where
\[
1/p(x)+1/p'(x)=1 \quad(x\in\R^n)
\]
(see e.g. \cite[Chap.~2]{CF13} or \cite[Chap.~3]{DHHR11}).
\subsection{The Hardy-Littlewood maximal function on variable Lebesgue spaces}
\label{subsec:M-VLE}
By $\cM(\R^n)$
denote the set of all measurable functions $p:\R^n\to[1,\infty]$
such that \eqref{eq:exponents} holds and the Hardy-Littlewood
maximal operator is bounded on $L^{p(\cdot)}(\R^n)$.

Assume that \eqref{eq:exponents} is fulfilled. Diening \cite{D04}
proved that if $p$ satisfies
\begin{equation}\label{eq:log-Hoelder}
|p(x)-p(y)|\le\frac{c}{\log(e+1/|x-y|)}\quad (x,y\in\R^n)
\end{equation}
for some $c>0$ independent of $x,y\in\R^n$
and $p$ is constant outside some ball, then $p\in\cM(\R^n)$. Further, the behavior
of $p$ at infinity was relaxed by Cruz-Uribe, Fiorenza, and Neugebauer \cite{CFN03,CFN04},
where it was shown that if $p$ satisfies \eqref{eq:log-Hoelder} and
there exists a $p_\infty>1$ such that
\begin{equation}\label{eq:Hoelder-infinity}
|p(x)-p_\infty|\le\frac{c}{\log(e+|x|)}\quad (x\in\R^n)
\end{equation}
with $c>0$ independent of $x\in\R^n$,
then $p\in\cM(\R^n)$. Following \cite[Section~4.1]{DHHR11}, we will
say that if conditions
\eqref{eq:log-Hoelder}--\eqref{eq:Hoelder-infinity} are fulfilled, then
$p$ is {\em globally log-H\"older continuous}. The class of all
globally log-H\"older continuous exponents will be denoted by $\cP^{\log}(\R^n)$.

Conditions \eqref{eq:log-Hoelder} and \eqref{eq:Hoelder-infinity} are optimal
for the boundedness of $M$ in the  pointwise sense; the corresponding
examples are contained in \cite{PR01} and \cite{CFN03}.
However, neither \eqref{eq:log-Hoelder} nor \eqref{eq:Hoelder-infinity}
is necessary for $p\in\cM(\R^n)$. Nekvinda \cite{N04} proved that if $p$
satisfies \eqref{eq:exponents}--\eqref{eq:log-Hoelder} and
\begin{equation}\label{eq:Nekvinda}
\int_{\R^n}|p(x)-p_\infty|c^{1/|p(x)-p_\infty|}\,dx<\infty
\end{equation}
for some $p_\infty>1$ and $c>0$, then $p\in\cM(\R^n)$. One can show that
\eqref{eq:Hoelder-infinity} implies \eqref{eq:Nekvinda}, but the converse,
in general, is not true. The corresponding example is constructed in \cite{CCF07}.
Nekvinda further relaxed condition \eqref{eq:Nekvinda} in \cite{N08}.
Lerner \cite{L05} (see also \cite[Example~5.1.8]{DHHR11})
showed that there exist discontinuous at zero or/and at infinity exponents,
which nevertheless belong to $\cM(\R^n)$. Thus, the class of exponents
in $\cP^{\log}(\R^n)$ satisfying \eqref{eq:exponents} is a proper subset
of the class $\cM(\R^n)$.

We will need the following remarkable result proved in \cite[Theorem~8.1]{D05}
(see also \cite[Theorem~5.7.2]{DHHR11}).
\begin{theorem}[Diening]
\label{th:Diening}
We have $p\in\cM(\R^n)$ if and only if $p'\in\cM(\R^n)$.
\end{theorem}
We refer to the recent monographs \cite{CF13,DHHR11} for further discussions
concerning the class $\cM(\R^n)$.
\subsection{Boundedness of pseudodifferential operators on variable Lebesgue spaces}
\label{subsec:PDO-VLE}
Combining Theorem~\ref{th:main} and Theorem~\ref{th:Diening}, we immediately
arrive at the following.
\begin{theorem}
Suppose $p\in\cM(\R^n)$. If $a$ belongs to one of the following symbol classes:
\begin{enumerate}
\item[{\rm(a)}]
the H\"ormander class
$S_{\rho,\delta}^{n(\rho-1)}$ with $0<\rho\le 1$ and $0\le\delta<1$;

\item[{\rm(b)}]
the Miyachi class
$S_{\rho,\delta}^{n(\rho-1)}(\varkappa,n)$ with $0\le\delta\le\rho\le 1$, $0\le\delta<1$,  and $\varkappa>0$;
\end{enumerate}
then $\Op(a)$ extends to a bounded operator on $L^{p(\cdot)}(\R^n)$.
\end{theorem}
Part (a) of the above theorem was obtained by the author and
Spitkovsky in \cite[Theorem~1.2]{KS13a}. Part (b) is new.
\begin{corollary}[Rabinovich-Samko]
Let $p\in\cP^{\log}(\R^n)$ satisfy \eqref{eq:exponents}. If $a\in S_{1,0}^0$,
then $\Op(a)$ extends to a bounded operator on $L^{p(\cdot)}(\R^n)$.
\end{corollary}
\begin{proof}
This statement immediately follows from part (a) of the previous result
because $\cP^{\log}(\R^n)$ is a (proper!) subset of $\cM(\R^n)$.
\end{proof}
This result was proved in \cite[Theorem~5.1]{RS08}.

\end{document}